\documentclass{amsart}

\usepackage[english]{babel}

\usepackage{amsmath, amsfonts, amssymb, amsthm, faktor}

\usepackage[all]{xy}
\usepackage{tikz}
\usetikzlibrary{math}
\usetikzlibrary{patterns}
\usepackage{caption}
\usepackage{subfig}

\usepackage{caption}
\usetikzlibrary{arrows,chains,matrix,positioning,scopes,decorations.pathreplacing,
decorations.pathmorphing,decorations.markings,arrows.meta}

\usepackage{aliascnt}
\usepackage[colorlinks, linktocpage, allcolors=black,breaklinks]{hyperref}

\usepackage{enumerate}

\usepackage{verbatim}

\theoremstyle{definition}
\newtheorem{theorem}{Theorem}[section]

\newtheorem{proposition}[theorem]{Proposition}

\newtheorem{question}[theorem]{Question}
\newtheorem{remark}[theorem]{Remark}

\newtheorem*{acknowledgments}{Acknowledgments}
\newcommand{\G}{\Gamma}

\newcommand{\acts}{\curvearrowright}

\begin{document}
\title[Borel Kernels]{Borel Kernels in Borel Directed Graphs}

\author[Ruijun Wang]{Ruijun Wang}
\address{School of Mathematical Sciences and School of pre-university, Dalian Minzu University}
\email{wangruijun@dlnu.edu.cn}
\thanks{}



\date{}

\maketitle

\begin{abstract}
We prove that there is a Borel quasi-kernel in any locally countable Borel directed graph with finite Borel chromatic number. We prove that the Borel chromatic number of a Borel directed graph with bounded out-degree $n$ is either infinite or less than or equal to $\frac{(n+1)(n+2)}{2}$. This is an alternative proof of Palamourdas' theorem.
\end{abstract}

\section{Introduction}
In the seminal paper \cite{KST}, Kechris, Solecki and Todorcevic initiated the study of descriptive combinatorics of Borel graphs. In particular, they proved that the Borel chromatic number of a Borel graph with bounded degree $n$ is less than or equal to $n+1$. Subsequently, in \cite{Marks}, Marks studied combinatorics of free products of two marked groups, and showed that for each $2\leq m\leq n+1$ there is a $n$-regular acyclic Borel graph with Borel chromatic number $m$. This showed the bound $n+1$ is optimal.

Kechris, Solecki and Todorcevic also considerd the Borel graph generated by $n$ many countable-to-$1$ functions, which is equivalent to considering a locally countable Borel directed graph with bounded out-degree $n$. They proved that the Borel chormatic number is either infinite or less than or equal to $3^n$, and asked a question, see \cite[Question 4.9]{KST}, whether the bound $3^n$ can be improved to $2n+1$. This question is motivated by a theorem of De Bruijn and Erd\"os in finite graph theory, which says the chormatic number of a finite directed graph is less than or equal to $2n+1$ where $n$ is the maximum number of out-degree. For example, the directed Schreier graph $\vec F(2^{\mathbb F_n})$ has in-degree $n$ and out-degree $n$, in \cite{Marks}, Marks proved it has Borel chromatic number $2n+1$. Subsequently, in \cite{Palamourdas}, Palamourdas proved that the bound can be improved to $2n+1$ if the $n$ functions are commutative, and can be improved to $\frac{(n+1)(n+2)}{2}$ in general case, and answered the question when $n=2$, and almost answered the question when $n=3$ (with $8$ instead of $7$). In \cite{Meehan}, Meehan and Palamourdas followed the idea of Palamourdas' thesis work and proved that the bound can be improved to $\frac{(n+1)(n+2)}{2}-2$ when $n\geq 4$. For an overview of the entire field of descriptive combinatorics we refer the reader to the survey \cite{KM} by Kechris and Marks, in particular, Section 5.3 talks about this question.

In this paper, we show an alternative proof of Palamourdas' theorem, see \cite[Lemma 5.1]{Palamourdas}, using Theorem \ref{thm:1.2}.

\begin{theorem}\label{thm:1.1}
    Let $D$ be a locally countable Borel directed graph with bounded out-degree $n$, then the Borel chromatic number $\chi_B(\Tilde{D})=\infty$ or $\chi_B(\Tilde{D})\leq\frac{(n+1)(n+2)}{2}$ where $\Tilde{D}$ is the underlying graph of $D$.
\end{theorem}

For a directed graph, a set of vertices $A$ is recurrent if $\rho(x,A)<\infty$ for all $x$ where $\rho$ is the directed distance $\rho(x,A)=\inf_{y\in A}\rho(x,y)$, and is recurrent with bounded gap $d$ if $\rho(x,A)\leq d$ for all $x$. An independent set of vertices is called a quasi-kernel (or semi-kernel) if it is recurrent with bounded gap $2$, and is called a kernel if it is recurrent with bounded gap $1$. In \cite{QuasiKernel}, Chv\'{a}tal and Lov\'{a}sz proved that there is a quasi-kernel for any finite directed graph. In \cite{Richardson}, Richardson proved that there is a kernel for any finite odd-dicycle-free graph. In \cite[Lemma 4.1]{CJMST20}, Conley, Jackson, Marks, Seward and Tucker-Drob proved that there is an independent and recurrent with bounded gap $3$ Borel set for the graph generated by a bounded-to-1 Borel function. In \cite{Higgins}, Higgins proved that there is an independent and recurrent with bounded gap $\chi_B(\Tilde{D})+1$ Borel set for any locally countable Borel directed graph with finite Borel chromatic number $\chi_B(\Tilde{D})<\infty$. In this paper, we prove that there is a quasi-kernel for any locally countable Borel directed graph with finite Borel chromatic number.

\begin{theorem}\label{thm:1.2}
    Let $D$ be a locally countable Borel directed graph with finite Borel chromatic number $\chi_B(\Tilde{D})<\infty$, then there is a Borel quasi-kernel in $D$.
\end{theorem}

The rest of the paper is organized as follows. In Section~\ref{sec:2} we fix the notation to be used throughout the paper. In Section~\ref{sec:3} we prove Theorem \ref{thm:1.2}. Then in Section~\ref{sec:4} we prove Theorem \ref{thm:1.1}. Finally, in Section~\ref{sec:5} we ask some questions.


\section{Preliminaries\label{sec:2}}
In this section we define our basic notions and fix notation. We use standard concepts and terminology from graph theory, which can be found in \cite{book}.

A {\em directed graph} (or {\em digraph}) $D$ is a pair $(V(D),A(D))$, where $V(D)$ is a set and $A(D)\subseteq V(D)\times V(D)$ is a non-reflexive binary relation on $V(D)$, i.e., $(x,x)\notin A(D)$ for all $x\in V(D)$. Here $V(D)$ is the set of {\em vertices} of $D$, and $A(D)$ is the set of {\em arcs} of $D$. If $a=(x,y)\in A(D)$, we also denote an arc $a=(x,y)$ by $x\to y$ and say that $y$ is an {\em out-neighbor} of $x$ and $x$ is an {\em in-neighbor} of $y$. For $x\in V(D)$, the {\em out-degree} is the number of out-neighbors of $x$, similarly, we define {\em in-degree}.

If $x,y\in V(D)$ are distinct, then a {\em directed path} $p$ from $x$ to $y$ is a sequence $x_0, x_1,\dots, x_k$ of distinct vertices of $D$ such that $x_0=x$, $x_k=y$, and for all $0\leq i<k$, $(x_i,x_{i+1})\in A(D)$; here $k$ is the {\em length} of the directed path $p$. The {\em directed distance} from $x$ to $y$ is the minimum of length of directed paths from $x$ to $y$.

For a directed graph $D$, the {\em chromatic number} is the chromatic number of its underlying graph $\Tilde{D}$ where $V(\Tilde{D})=V(D)$ and $(x,y)\in A(\Tilde{D})\Longleftrightarrow (x,y)\in A(D)\text{ or }(y,x)\in A(D)$. When $V(D)$ is a standard Borel space, and $D$ is a Borel subset of $V(D)\times V(D)$, we say that $D$ is a Borel directed graph, we can define {\em Borel} proper colorings of $G$, and define its {\em Borel chromatic number}, denoted as $\chi_B(\Tilde{D})$.

For a directed graph $D$, a set of vertices $A$ is {\em independent} if $A$ is independent in its underlying graph $\Tilde{D}$. A set of vertices $A$ is {\em recurrent} if $\rho(x,A)<\infty$ for all $x$ where $\rho$ is the directed distance $\rho(x,A)=\inf_{y\in A}\rho(x,y)$, and is {\em recurrent with bounded gap $d$} if $\rho(x,A)\leq d$ for all $x$. An independent set of vertices is called a {\em quasi-kernel} (or {\em semi-kernel}) if it is recurrent with bounded gap $2$, and is called a {\em kernel} if it is recurrent with bounded gap $1$. When $D$ is a Borel directed graph, we may consider {\em Borel quasi-kernel} and {\em Borel kernel}.

Let $X$ be a set and let $f_1,\cdots,f_n:X\to X$ be functions, we define the {\em directed graph $D=D_{f_1,\cdots,f_n}$ generated by functions}. $V(D)=X$ and $$(x,y)\in A(D)\Longleftrightarrow \exists\ 1\leq i\leq n\ y=f_i(x)\text{ and }x\neq y.$$ The {\em graph $G$ generated by functions} is the underlying graph $G=\Tilde{D}$. By Borel uniformization, see \cite[Theorem 18.10]{Kechris95}, $D$ is a Borel directed graph generated by $n$ many countable-to-1 functions if and only if $D$ is locally countable Borel directed graph with bounded out-degree $n$.

An important class of Borel graphs consists of the graphs induced by actions of finitely generated groups. A {\em marked group} is a pair $(\Gamma, S)$, where $\Gamma$ is a group and $S$ is a finite generating set of $\Gamma$. Usually we also require $1_\Gamma\not\in S$, and in this paper we do not assume $S$ is symmetric. 

When there is an action of a marked group $\Gamma$ on a set $X$, the {\em directed Schreier graph} of the action $\Gamma\acts X$ on $X$, denoted $\vec D(\Gamma, S,X)$, is defined by
$$ V(\vec D(\Gamma, S,X))=X $$
and
$$ A(\vec D(\Gamma, S,X))=\{(x,y)\in X^2\,:\, \exists s\in S\ y=s\cdot x\}. $$

The Schreier graph will be particularly nice when the action is free. Let $F(\Gamma, X)=\{ x \in X\colon  \forall g \neq 1_\G \ (g\cdot x \neq x)\}$ be the {\em free part} of the action $\Gamma\curvearrowright X$. $F(\Gamma, X)$ is an invariant set and the induced action of $\Gamma$ on $F(\Gamma, X)$ is free.  We denote by $\vec F(\Gamma,S,X)$ the directed Schreier graph on $F(\Gamma, X)$. If the group $\Gamma$ and its action on $X$ are understood and the generating set $S$ is standard or otherwise understood, we simply write $\vec F(X)$ for the directed Schreier graph $\vec F(\Gamma,S, X)$.

The Bernoulli shift action induces a natural Schreier graph. Let $A$ be a finite set with the discrete topology. Assume $|A|\geq 2$. For a countable discrete group $\G$, the space $A^\G$ is equipped with the usual product topology and is homeomorphic to the Cantor space $2^\mathbb N=\{0,1\}^{\mathbb{N}}$, which is compact Polish.

The {\em Bernoulli shift action} $\G\curvearrowright A^\G$ is defined by
\[
(g \cdot x)(h)= x(hg)
\]
for $g,h \in \G$ and $x \in A^\G$. 
This action is continuous. The free part $F(\Gamma, A^\Gamma)$ is an invariant dense $G_\delta$ subset of $A^\Gamma$, hence is a Polish space. 

\section{Proof of Theorem \ref{thm:1.2}}\label{sec:3}

For the readers' convenience, we state Theorem \ref{thm:1.2} again.

\begin{theorem}
    Let $D$ be a locally countable Borel directed graph with finite Borel chromatic number $\chi_B(\Tilde{D})<\infty$, then there is a Borel quasi-kernel in $D$, i.e., there is a Borel set of vertices $M$ such that
    \begin{enumerate}[(1)]
        \item $M$ is independent;
        \item $\rho(x,M)\leq 2$ for all vertices $x$ where $\rho$ is the directed distance.
    \end{enumerate}
\end{theorem}
\begin{proof}
    We prove by induction on $\chi_B(\Tilde{D})$, when $\chi_B(\Tilde{D})=1$, take $M=V(D)$. Suppose the theorem holds for $\chi_B(\Tilde{D})\leq k$, now assume that $\chi_B(\Tilde{D})=k+1$, fix a Borel proper $k+1$-coloring $c$.

    Let $A=c^{-1}(\{0\})$ be the vertices colored by $0$. Let $T=\{x\notin A:\ \text{there is no }y\in A\ x\to y\}$. $T$ could be empty, in which case, take $M=A$. Assume now $T$ is not empty, $T$ is Borel because $D$ is locally countable. Consider the induced Borel subgraph $D[T]$, it has a Borel proper $k$-coloring, by inductive hypothesis applied to $D[T]$, there is a Borel recurrent set $T'\subseteq T$ such that $\rho(x,T)\leq 2$ for all $x\in T$ (where $\rho$ is the directed distance in $D[T]$) and $T'$ is independent in $\Tilde{D}[T]$ and thus independent in $\Tilde{D}$.

    Let $A'=\{x\in A:\ \text{there is }y\in T'\ \text{that }x\to y\}$. $A'$ is Borel because $D$ is locally countable. $A'$ could be empty, but it does not make any difference. Take $M=(A\backslash A')\cup T'$.

    To see that $M$ is independent in $\Tilde{D}$, assume towards a contradiction that there are $x,y\in M\ \text{that }x\Tilde{D}y$. Because of independence of $A$ and $T'$, without loss of generality, we assume that $x\in A\backslash A'$ and $y\in T'$. By the definition of $A'$, there is no arc from $x$ to $y$. By the definition of $T$, there is no arc from $y$ to $x$.

    We claim that $\rho (x,M)\leq 2$ for all $x$. If $x\in M$, then we are done. Now let $x\notin A\backslash A'$ and $x\notin T'$.

    (Case 1) Suppose $x\notin T$.

    \ \ (Case 1.1) If $x\in A'$, then there is an arc from $x$ to $T'\subseteq M$.

    \ \ (Case 1.2) If $x\notin A$, then there is $y\in A$ such that there is an arc from $x$ to $y$ by the definition of $T$,

    \ \ \ \ (Case 1.2.1) if $y\notin A'$ then $y\in M$;

    \ \ \ \ (Case 1.2.2) if $y\in A'$ then there is an arc from $y$ to $T'\subseteq M$.

    (Case 2) Suppose $x\in T$, then $x\in T\backslash T'$, by the inductive hypothesis $\rho(x,T')\leq 2$, so $\rho(x,M)\leq 2$.

     This completes the induction.
\end{proof}

In the above theorem, the assumption $\chi_B(\Tilde{D})<\infty$ cannot be removed. Galvin and Prikry showed there is a locally finite Borel directed graph generated by a Borel function with infinite Borel chromatic number, see \cite[Theorem 19.11]{Kechris95}. Let $V=[\omega]^\omega$ be the infinite sets on $\omega$, define $f(x)=x\backslash\min\{x\}$, then the Borel graph $G_f$ has infinite Borel chromatic number, and has no independent and recurrent Borel set.

\begin{proposition}\label{SingleFunction}
    Let $D$ be the Borel directed graph generated by a countable-to-$1$ Borel function, then the following are equivalent.
    \begin{enumerate}[(1)]
        \item $\chi_B(\Tilde{D})<\infty$;
        \item $\chi_B(\Tilde{D})\leq 3$;
        \item there is an independent and recurrent with bounded gap Borel set;
        \item there is an independent and recurrent with bounded gap $2$ Borel set;
        \item there is an independent and recurrent Borel set.
    \end{enumerate}
\end{proposition}
\begin{proof}
    $(2)\Longrightarrow (1)$, $(b)\Longrightarrow(a)$, $(4)\Longrightarrow (3)\Longrightarrow(5)$ are trivial.
    
    $(1)\Longrightarrow(4)$ is by Theorem \ref{thm:1.2}.
    
    

    For $(5)\Longrightarrow (2)$,  we color the independent and recurrent set $A$ by color $0$, and color its in-neighborhood $N^-(A)$ by color $1$, and color their in-neighborhood $N^-(A\cup N^-(A))$ by color $2$, we use color $1$ and color $2$ alternatively, until all vertice are colored.
\end{proof}

\begin{remark}
    We remark that a finite-to-1 Borel function graph could have finite Borel chromatic number and infinite Borel asymptotic dimension. For example, let $Y=[\omega]^\omega$ and $s:Y\to Y$ be $s(y)=y\backslash\{\min y\}$. It is the shift graph. Let $X=Y\times \{0,1\}$ and $f(y,0)=(y,1)$, $f(y,1)=(s(y),0)$. Then $f$ is finite-to-1, $G_f$ has a Borel proper 2-coloring. By \cite[Theorem 1.2]{GrebikHiggins}, one can show that there is no Borel $2$-forward-independent hitting set, so it is of infinite Borel asymptotic dimension.
\end{remark}

\section{Proof of Theorem \ref{thm:1.1}}\label{sec:4}

\begin{theorem}\label{Dominate}
    Let $D$ be a locally countable Borel directed graph with finite Borel chromatic number $\chi_B(\Tilde{D})<\infty$ and with bounded out degree $n$, then there is a Borel set $M$ such that $\chi_B(\Tilde{D}[M])\leq n+1$ and $\rho(x,M)\leq 1$ for all $x$ where $\rho$ is the directed distance on $D$.
\end{theorem}
\begin{proof}
    We prove by induction on $n$. When $n=1$, by Theorem \ref{thm:1.2}, there is an independent Borel set $M'$ such that $\rho(x,M')\leq 2$. Note that the in-neighborhood $N^-(M')=\{x:\rho(x,M')=1\}$ is an independent Borel set, so $M=M'\cup N^-(M')$ is as required.

    Suppose the theorem holds for $n-1$, when $D$ has bounded out-degree $n$, by Theorem \ref{thm:1.2}, there is an independent Borel set $M'$ such that $\rho(x,M')\leq 2$. Let $A=\{x:\rho(x,M')=2\}=V(D)\backslash(M'\cup N^{-1}(M'))$ where $N^{-1}(M')$ is the in-neighborhood of $M'$. Consider the induced subgraph $D[A]$, it has bounded out-degree $n-1$, since each vertex in $A$ has an out-going arc to $N^{-1}(M)$, by inductive hypothesis applied to $A$, there is a Borel set $M_0$ with $\chi_B(\Tilde{D}[M_0])\leq n$ such that $\rho(x,M_0)\leq 1$ for all $x\in A$. Then $M=M'\cup M_0$ is as required. This completes the induction.
\end{proof}

Now we are ready to prove Theorem \ref{thm:1.1}. For the readers' convenience, we state Theorem \ref{thm:1.1} again.

\begin{theorem}
    Let $D$ be a locally countable Borel directed graph with bounded out-degree $n$, then the Borel chromatic number $\chi_B(\Tilde{D})=\infty$ or $\chi_B(\Tilde{D})\leq\frac{(n+1)(n+2)}{2}$ where $\Tilde{D}$ is the underlying graph of $D$.
\end{theorem}

\begin{proof}
    We prove by induction on $n$. When $n=1$, it is proved by Proposition \ref{SingleFunction}. 
    
    Suppose the theorem holds for $n-1$, when $D$ has bounded out-degree $n$, by Theorem \ref{Dominate}, there is a Borel set $M$ such that $\chi_B(\Tilde{D}[M])\leq n+1$ and $\rho(x,M)\leq 1$ for all $x$ where $\rho$ is the directed distance on $D$. Consider the induced subgraph of the complement $D'=D[(V(D)\backslash M)]$, the subgraph has bounded out-degree $n-1$, by inductive hypothesis applied to $D'$, it has Borel chromatic number $\chi_B(\Tilde D')=\infty$ or $\chi_B(\Tilde D')\leq\frac{n(n+1)}{2}$, now we can color $M$ by $n+1$ many new different colors. This completes the induction.
\end{proof}

\section{Miscellaneous}\label{sec:5}
During the submission of this paper, the author of this paper proved that for any finite digraph $D$ with out-bounded degree $\Delta^+(D)=\Delta^+>1$, there is a $\Delta^+$-partition such that the induced subgraph of each partition is odd-dicycle-free. As a corollary, the digraph has a $\Delta^+$-partite domination. So, Theorem \ref{Dominate} can be improved for finite graphs. But we still don't know if it can be improved for Borel graphs. Also, there is an example showing that it is sharp for finite graphs when $n=3$, Paley tournament of $7$ vertices,

$$V(\text{PTr}(7))=\mathbb Z_7,xy\in A(\text{PTr}(7))\Longleftrightarrow y-x\equiv1\text{ or }2\text{ or }4(\text{mod }7).$$

If $M$ is a domination in $\text{PTr}(7)$, i.e., $\rho(x,M)\leq 1$ for all $x$, then $|M|\geq 3$.

\begin{question}
    Let $D$ be a locally countable Borel directed graph with finite Borel chromatic number $\chi_B(\Tilde{D})<\infty$ and with bounded out degree $2$, is there a Borel set $M$ such that $\chi_B(\Tilde{D}[M])\leq 2$ and $\rho(x,M)\leq 1$ for all $x$ where $\rho$ is the directed distance on $D$?
\end{question}

If it is true, then it will improve the bound in \cite{Meehan}. The directed Schreier graph $\vec F(2^{\Gamma*\Delta})$ are potential candidates for counter-examples where $\Gamma,\Delta=\mathbb Z\text{ or }\mathbb Z_n$.

In \cite{Richardson}, Richardson proved that there is a kernel for any finite odd-dicycle-free graph. But it is not true in general for Borel graphs. For example, an easy ergodicity argument shows that there is no measurable kernel in the directed Schreier graph $\vec F(2^{\mathbb Z^2})$. For an another example, the directed Schreier graph $\vec F(2^{\mathbb F_n})$, or any Borel directed graph $D$ with Borel chromatic number $5$ and with bounded out-degree $2$ has no Borel kernel, assume towards a contradiction, there is a Borel kernel $A$, then the induced subgraph of the complement $D[(V(D)\backslash A)]$ has bounded out-degree $1$, and thus has Borel chromatic number less than or equal to $3$, so $\chi_B(\Tilde{D})\leq 4$, a contradiction. But this argument does not work for measure settings.

\begin{question}
    Is there a measurable kernel in the directed Schreier graph $\vec F(2^{\mathbb F_n})$?
\end{question}

Similar to proposition \ref{SingleFunction}, one can show that the Borel graph generated by a 2-1 Borel function has a $4$ proper edge-coloring, but we do not know whether this bound can be improved.

\begin{question}
    Is there a 2-to-1 Borel function such that the Borel graph generated by it has Borel edge chromatic number $4$?
\end{question}

\begin{acknowledgments}
The author would like to thank Cecelia Higgins for reading early drafts of this paper and helpful discussions. This paper is partially selected from the author's PhD thesis, the author would like to thank co-advisors Longyun Ding and Su Gao for many years of advice.
\end{acknowledgments}

\end{document}